\DocumentMetadata{
  pdfstandard = ua-2,
  lang        = en-US,
  tagging=on,
tagging-setup = {math/alt/use,
		  math/setup={mathml-AF,mathml-SE},
                   extra-modules={verbatim-mo},
                   table/header-rows=1}
}
\documentclass{article}
\usepackage{tikz}
\usepackage{pgfplots}
\usetikzlibrary{pgfplots.fillbetween}
\pgfplotsset{compat=1.18}
\definecolor{darkgreen}{RGB}{0,100,0}

\usepackage{fullpage} 
\usepackage{amsmath}
\usepackage{amsthm}
\usepackage{amssymb}
\usepackage{cite}
\usepackage[normalem]{ulem}

\usepackage[hidelinks]{hyperref}
\newtheorem{theorem}{Theorem}[section]
\newtheorem{corollary}[theorem]{Corollary}

\newtheorem{lemma}[theorem]{Lemma}

\newtheorem{claim}[theorem]{Claim}

\newcommand{\floor}[1]{\left\lfloor #1 \right\rfloor}

\let\epsilon\varepsilon

\tikzset{unlabeled_vertex/.style={inner sep=1.7pt, outer sep=0pt, circle, fill}}
\begin{document}

\title{Maximizing $K_r + I_r$ in graphs with fixed edge density
}
\author{
J\'ozsef Balogh~\thanks{Department of Mathematics, University of Illinois Urbana-Champaign, Urbana, IL, USA, and Extremal Combinatorics and Probability Group (ECOPRO), Institute for Basic Science (IBS), Daejeon, South Korea. Email: \texttt{jobal@illinois.edu}. Supported by NSF grants RTG DMS-1937241, FRG DMS-2152488, the Arnold O. Beckman Research Award (UIUC Campus Research Board RB 24012), the Simons Collaboration grant SFI-MPS-TSM-00013107, and the Institute for Basic Science (IBS-R029-C4).}
\and
Andrzej Grzesik~\thanks{
Faculty of Mathematics and Computer Science, Jagiellonian University, ul.~Prof.~St.~\L ojasiewicza 6, 30-348 Krak\'{o}w, Poland.
E-mail: \texttt{andrzej.grzesik@uj.edu.pl}.
Research of this author was partially supported by the National Science Centre grant 2021/42/E/ST1/00193.}
\and
Bernard Lidick\'{y}~\thanks{Department of Mathematics, Iowa State University, Ames, IA. E-mail: \texttt{lidicky@iastate.edu}. Research of this author was partially supported by NSF FRG DMS-2152490, DMS-2554129, the Simons Foundation TSM-00013439 and Scott Hanna Professorship.}
\and
Theodore Molla~\thanks{
Department of Mathematics and Statistics, University of
South Florida, Tampa, FL 33620, USA. E-mail: \texttt{molla@usf.edu}.
Research of this author was partially supported by NSF Award DMS-2154313.}
\and
Dhruv Mubayi~\thanks{
Department of Mathematics, Statistics and Computer Science, University of Illinois, Chicago, IL 60607. E-mail: \texttt{mubayi@uic.edu}. Research of this author was partially supported by NSF Awards DMS-2153576 and DMS-2552740.
}
\and
Jan Volec~\thanks{
Department of Theoretical Computer Science, Faculty of Information Technology,
Czech Technical University in Prague, Th\'akurova 9, Prague, 160 00, Czech Republic.
E-mail: \texttt{jan@ucw.cz}.
Research of this author was partially supported by the grant 23-06815M of the Grant Agency of the Czech
Republic.
}
}
\date{\today}

\maketitle

\begin{abstract}
For every integer $r\ge4$, and $\rho \in [0,1]$, we asymptotically determine the maximum proportion of $r$-element sets of vertices that induce either a clique or an independent set in a large graph with density $\rho$.  This generalizes a result of Olpp for $r=3$. 

After the initial idea for the main proof was found by the authors, various AI models were used to streamline the argument and perform the calculations necessary for completion of the proof.
\end{abstract}

MSC2020 05C35,  05C75

\section{Introduction}
 For a graph $G$, write $v(G)=|V(G)|$, $e(G) = |E(G)|$ and
$\rho(G)=e(G)/\binom{v(G)}{2}$ for its edge density.  For a graph $F$, let
$F(G)$ be the number of $v(F)$-subsets of $V(G)$ that induce a copy of $F$,
and write
\[
 t_{\mathrm{ind}}(F,G)=\frac{F(G)}{\binom{v(G)}{v(F)}}
\]
for the induced density of $F$ in $G$.  Denote by $K_r$  the complete graph on
$r$ vertices and $I_r$ the independent set on $r$ vertices.

A classical problem in extremal graph theory asks for the maximum (or minimum) proportion of copies of a  graph  (or a linear combination of some collection of graphs) in a large graph $G$ with given edge density. See~\cite{A, BP, DS, E, GNPV, N, Raz, R, lili} for some papers on these topics. 
For example,
  Goodman~\cite{Goodman} determined the minimum of $K_3(G)+I_3(G)$, and  he asked for
the maximum of $K_3(G)+I_3(G)$ among graphs $G$ with a prescribed number of edges.  Olpp~\cite{Olpp}
determined this maximum, showing that, up to the adjustments needed to obtain
the prescribed edge count, it is attained by the disjoint union of a clique
and isolated vertices or by its complement.  We prove the analogous
asymptotic result for $K_r(G)+I_r(G)$ for every fixed $r\ge3$.

\newcommand{\K}{L}
Let $\K=\K(n,\gamma)$ be the $n$-vertex graph consisting of a clique on
$\floor{\gamma n}$ vertices and $n-\floor{\gamma n}$ isolated vertices, and
let $\overline \K=\overline \K(n,\gamma)$ be its complement; see
Figure~\ref{fig:construction}.  The edge density of $L$ is
$\gamma^2+O(n^{-1})$, and of $\overline L$ is
$1-\gamma^2+O(n^{-1})$.

Moreover,
$K_r(\overline \K)+I_r(\overline \K)=K_r(\K)+I_r(\K)$ equals
\[
  \binom{\floor{\gamma n}}{r}+\binom{n-\floor{\gamma n}}{r}
  +\floor{\gamma n}\binom{n-\floor{\gamma n}}{r-1}
  \ge \left(\gamma^r+(1-\gamma)^r+r\gamma(1-\gamma)^{r-1}\right)\binom{n}{r}
  - O_r(n^{r-1}).
\]

We show that the asymptotic maximum of $K_r(G)+I_r(G)$ over all $n$-vertex
graphs $G$ with edge density $\rho+o(1)$ is attained by
$\K(n,\sqrt\rho)$ when $\rho>1/2$ and by
$\overline \K(n,\sqrt{1-\rho})$ when $\rho\le1/2$; see
Figure~\ref{fig:graph}. All asymptotic notation in the theorem statement below is taken as $n \rightarrow \infty$.

\begin{theorem}\label{thm:main}
Fix an integer $r\ge3$ and a real $\rho \in [0,1]$.
If $G_n$ is an $n$-vertex graph with edge density $\rho+o(1)$, then it holds that
\[
 \frac{K_r(G_n)+I_r(G_n)}{\binom nr}\le
 \gamma^r+(1-\gamma)^r+r\gamma(1-\gamma)^{r-1}
 +o(1),
 \quad\mbox{where }
 \gamma=\begin{cases}
    \sqrt{1-\rho} \quad\mbox{for $\rho\le \frac{1}{2} $, and}\\
    \sqrt\rho \quad \mbox{otherwise}.
 \end{cases}
\]
\end{theorem}

\begin{figure}[h!]
\begin{center}
    \begin{tikzpicture}[alt={Independent set on (1-gamma) n vertices and a clique on gamma n vertices.}]
        \draw
        \foreach \i in {1,...,10}{
        (36*\i:1) node[unlabeled_vertex](x\i){}
        };
        \draw
        \foreach \j in {1,...,8}{
        (3,0)+(90-25+25*\j:1) node[unlabeled_vertex](y\j){}
        }
        ;
        \foreach \i in {2,...,8}{
          \pgfmathtruncatemacro{\imax}{\i-1}
          \foreach \j in {1,...,\imax}{
            \draw[opacity=0.3] (y\i)--(y\j);
          }
        }
    \draw (0,0) node[fill=white,rectangle, inner sep=1pt]{$(1 - \gamma)n$};
    \draw (3,0) node[fill=white,rectangle, inner sep=1pt]{$\gamma n$};
    \draw(1.5,-1) node[below]{(a)};
    \end{tikzpicture}
    \hskip 2cm
    \begin{tikzpicture}[alt={A complement of the other picture.}]
        \draw
        \foreach \i in {1,...,10}{
        (36*\i:1) node[unlabeled_vertex](x\i){}
        };
        \draw
        \foreach \j in {1,...,8}{
        (3,0)+(90-25+25*\j:1) node[unlabeled_vertex](y\j){}
        }
        ;
        \foreach \i in {1,...,10}{
            \foreach \j in {1,...,8}{
\draw[opacity=0.3] (x\i)--(y\j);
            }
        }
        \foreach \i in {2,...,10}{
          \pgfmathtruncatemacro{\imax}{\i-1}
          \foreach \j in {1,...,\imax}{
            \draw[opacity=0.3] (x\i)--(x\j);
          }
        }
    \draw (0,0) node[fill=white,rectangle, inner sep=1pt]{$(1 - \gamma)n$};
    \draw (3,0) node[fill=white,rectangle, inner sep=1pt]{$\gamma n$};
    \draw(1.5,-1) node[below]{(b)};
    \end{tikzpicture}
\end{center}
\caption{The extremal constructions: (a) the disjoint union of a clique of
size approximately $\gamma n$ and $(1-\gamma)n$ isolated vertices; (b) the complement of (a).}
\label{fig:construction}
\end{figure}
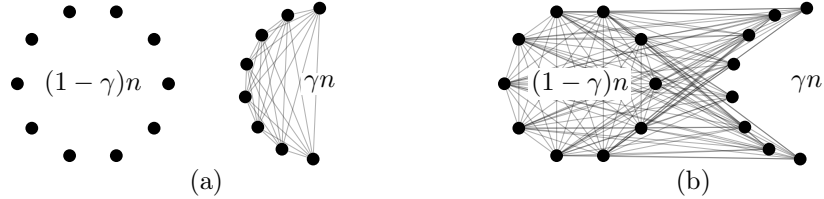

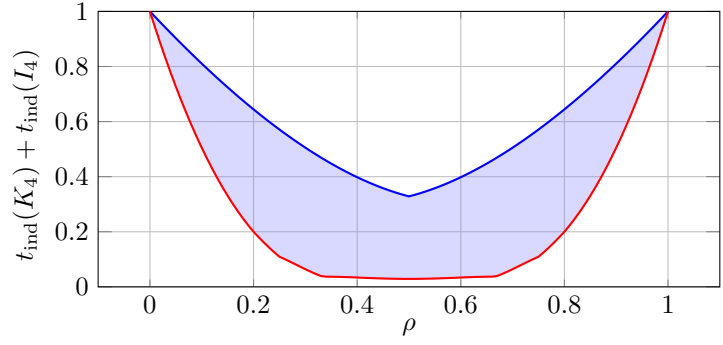
\begin{figure}
\begin{center}
\begin{tikzpicture}[alt={Plot of upper and lower bounds.},scale=0.8]
  \begin{axis}[
    xlabel={$\rho$},
    ylabel={$t_{\mathrm{ind}}(K_4)+t_{\mathrm{ind}}(I_4)$},
    ylabel style={
        at={(axis description cs:-0.12,0.6)},
        anchor=south,
    },
    ymin=0,ymax=1.0,
    grid=major,
    domain=0:1,
    xscale=1.5,
    yscale=0.8
  ]

    \addplot [name path=upper, blue, thick] coordinates {
      (0.000000,1.000000)
(.0100000000,   9.8010050e-01)
(.0200000000,   9.6040403e-01)
(.0300000000,   9.4091365e-01)
(.0400000000,   9.2163249e-01)
(.0500000000,   9.0256370e-01)
(.0600000000,   8.8371051e-01)
(.0700000000,   8.6507617e-01)
(.0800000000,   8.4666400e-01)
(.0900000000,   8.2847739e-01)
(.1000000000,   8.1051975e-01)
(.1100000000,   7.9279457e-01)
(.1200000000,   7.7530539e-01)
(.1300000000,   7.5805582e-01)
(.1400000000,   7.4104953e-01)
(.1500000000,   7.2429023e-01)
(.1600000000,   7.0778173e-01)
(.1700000000,   6.9152790e-01)
(.1800000000,   6.7553265e-01)
(.1900000000,   6.5980000e-01)
(.2000000000,   6.4433402e-01)
(.2100000000,   6.2913887e-01)
(.2200000000,   6.1421878e-01)
(.2300000000,   5.9957806e-01)
(.2400000000,   5.8522112e-01)
(.2500000000,   5.7115242e-01)
(.2600000000,   5.5737656e-01)
(.2700000000,   5.4389819e-01)
(.2800000000,   5.3072207e-01)
(.2900000000,   5.1785307e-01)
(.3000000000,   5.0529615e-01)
(.3100000000,   4.9305637e-01)
(.3200000000,   4.8113892e-01)
(.3300000000,   4.6954909e-01)
(.3400000000,   4.5829228e-01)
(.3500000000,   4.4737403e-01)
(.3600000000,   4.3680000e-01)
(.3700000000,   4.2657598e-01)
(.3800000000,   4.1670791e-01)
(.3900000000,   4.0720184e-01)
(.4000000000,   3.9806401e-01)
(.4100000000,   3.8930079e-01)
(.4200000000,   3.8091872e-01)
(.4300000000,   3.7292450e-01)
(.4400000000,   3.6532502e-01)
(.4500000000,   3.5812733e-01)
(.4600000000,   3.5133871e-01)
(.4700000000,   3.4496659e-01)
(.4800000000,   3.3901866e-01)
(.4900000000,   3.3350280e-01)
(.5000000000,   3.2842713e-01)
(.5100000000,   3.3350280e-01)
(.5200000000,   3.3901866e-01)
(.5300000000,   3.4496659e-01)
(.5400000000,   3.5133871e-01)
(.5500000000,   3.5812733e-01)
(.5600000000,   3.6532502e-01)
(.5700000000,   3.7292450e-01)
(.5800000000,   3.8091872e-01)
(.5900000000,   3.8930079e-01)
(.6000000000,   3.9806401e-01)
(.6100000000,   4.0720184e-01)
(.6200000000,   4.1670791e-01)
(.6300000000,   4.2657598e-01)
(.6400000000,   4.3680000e-01)
(.6500000000,   4.4737403e-01)
(.6600000000,   4.5829228e-01)
(.6700000000,   4.6954909e-01)
(.6800000000,   4.8113892e-01)
(.6900000000,   4.9305637e-01)
(.7000000000,   5.0529615e-01)
(.7100000000,   5.1785307e-01)
(.7200000000,   5.3072207e-01)
(.7300000000,   5.4389819e-01)
(.7400000000,   5.5737656e-01)
(.7500000000,   5.7115242e-01)
(.7600000000,   5.8522112e-01)
(.7700000000,   5.9957806e-01)
(.7800000000,   6.1421878e-01)
(.7900000000,   6.2913887e-01)
(.8000000000,   6.4433402e-01)
(.8100000000,   6.5980000e-01)
(.8200000000,   6.7553265e-01)
(.8300000000,   6.9152790e-01)
(.8400000000,   7.0778173e-01)
(.8500000000,   7.2429023e-01)
(.8600000000,   7.4104953e-01)
(.8700000000,   7.5805582e-01)
(.8800000000,   7.7530539e-01)
(.8900000000,   7.9279457e-01)
(.9000000000,   8.1051975e-01)
(.9100000000,   8.2847739e-01)
(.9200000000,   8.4666400e-01)
(.9300000000,   8.6507617e-01)
(.9400000000,   8.8371051e-01)
(.9500000000,   9.0256370e-01)
(.9600000000,   9.2163249e-01)
(.9700000000,   9.4091365e-01)
(.9800000000,   9.6040403e-01)
(.9900000000,   9.8010050e-01)
(1.0000000000,   1.0000000e+00)
    };
    \addplot [name path=lower, red, thick] coordinates {
      (0.000000,1.000000)
(.0100000000,  9.4109500e-01)
(.0200000000,  8.8436000e-01)
(.0300000000,  8.2976500e-01)
(.0400000000,  7.7728000e-01)
(.0500000000,  7.2687500e-01)
(.0600000000,  6.7852000e-01)
(.0700000000,  6.3218500e-01)
(.0800000000,  5.8784000e-01)
(.0900000000,  5.4545500e-01)
(.1000000000,  5.0500000e-01)
(.1100000000,  4.6644500e-01)
(.1200000000,  4.2976000e-01)
(.1300000000,  3.9491500e-01)
(.1400000000,  3.6188000e-01)
(.1500000000,  3.3062500e-01)
(.1600000000,  3.0112000e-01)
(.1700000000,  2.7333500e-01)
(.1800000000,  2.4724000e-01)
(.1900000000,  2.2280500e-01)
(.2000000000,  2.0000000e-01)
(.2100000000,  1.8001850e-01)
(.2200000000,  1.6084517e-01)
(.2300000000,  1.4259655e-01)
(.2400000000,  1.2539670e-01)
(.2500000000,  1.0937500e-01)
(.2600000000,  1.0086907e-01)
(.2700000000,  9.1841443e-02)
(.2800000000,  8.2614254e-02)
(.2900000000,  7.3317440e-02)
(.3000000000,  6.4131750e-02)
(.3100000000,  5.5253076e-02)
(.3200000000,  4.6887603e-02)
(.3300000000,  3.9324428e-02)
(.3400000000,  3.7000722e-02)
(.3500000000,  3.6785876e-02)
(.3600000000,  3.6427543e-02)
(.3700000000,  3.5963328e-02)
(.3800000000,  3.5400953e-02)
(.3900000000,  3.4741998e-02)
(.4000000000,  3.3932304e-02)
(.4100000000,  3.3031839e-02)
(.4200000000,  3.2272817e-02)
(.4300000000,  3.1615651e-02)
(.4400000000,  3.0979251e-02)
(.4500000000,  3.0385631e-02)
(.4600000000,  2.9848537e-02)
(.4700000000,  2.9461928e-02)
(.4800000000,  2.9102671e-02)
(.4900000000,  2.8890183e-02)
(.5000000000,  2.8933417e-02)
(.5100000000,  2.8890186e-02)
(.5200000000,  2.9102671e-02)
(.5300000000,  2.9461928e-02)
(.5400000000,  2.9848537e-02)
(.5500000000,  3.0385630e-02)
(.5600000000,  3.0979250e-02)
(.5700000000,  3.1615651e-02)
(.5800000000,  3.2272818e-02)
(.5900000000,  3.3031838e-02)
(.6000000000,  3.3932306e-02)
(.6100000000,  3.4741998e-02)
(.6200000000,  3.5400947e-02)
(.6300000000,  3.5963322e-02)
(.6400000000,  3.6427543e-02)
(.6500000000,  3.6785876e-02)
(.6600000000,  3.7000724e-02)
(.6700000000,  3.9324427e-02)
(.6800000000,  4.6887597e-02)
(.6900000000,  5.5253076e-02)
(.7000000000,  6.4131749e-02)
(.7100000000,  7.3317438e-02)
(.7200000000,  8.2614254e-02)
(.7300000000,  9.1841443e-02)
(.7400000000,  1.0086907e-01)
(.7500000000,  1.0937500e-01)
(.7600000000,  1.2539670e-01)
(.7700000000,  1.4259655e-01)
(.7800000000,  1.6084517e-01)
(.7900000000,  1.8001850e-01)
(.8000000000,  2.0000000e-01)
(.8100000000,  2.2280500e-01)
(.8200000000,  2.4724000e-01)
(.8300000000,  2.7333500e-01)
(.8400000000,  3.0112000e-01)
(.8500000000,  3.3062500e-01)
(.8600000000,  3.6188000e-01)
(.8700000000,  3.9491500e-01)
(.8800000000,  4.2976000e-01)
(.8900000000,  4.6644500e-01)
(.9000000000,  5.0500000e-01)
(.9100000000,  5.4545500e-01)
(.9200000000,  5.8784000e-01)
(.9300000000,  6.3218500e-01)
(.9400000000,  6.7852000e-01)
(.9500000000,  7.2687500e-01)
(.9600000000,  7.7728000e-01)
(.9700000000,  8.2976500e-01)
(.9800000000,  8.8436000e-01)
(.9900000000,  9.4109500e-01)
(1.0000000000,  1.0000000e+00)
    };

    \addplot [blue!50, fill opacity=0.3] fill between[of=upper and lower];

  \end{axis}
\end{tikzpicture}
\end{center}
\caption{
The extremal profile of
$t_{\mathrm{ind}}(K_4)+t_{\mathrm{ind}}(I_4)$.  The blue curve is the
asymptotic upper envelope proved in this paper; the red curve shows values suggested by
numerical experiments for the lower envelope, whose exact value is unknown.}
\label{fig:graph}
\end{figure}

{\bf Generative AI statement:}
The extremal example in Theorem~\ref{thm:main} has a two-part structure. Before using generative AI, the authors had a proof through the reduction to
a three-part structure, and the remaining step was to solve a constrained optimization problem involving a two-variable polynomial of degree $r$ (see Lemma~\ref{lem:maximizer3_domain} for the program).  Models including Gemini Pro and ChatGPT supplied
initial optimization arguments for completing this optimization and hence the three-part case; some
suggestions were incorrect and were discarded.  The final manuscript uses
AI-originated optimization arguments both for this final step and for a new
proof of the reduction from many parts to three.  Through iteration with the
authors, Prism, using GPT-5.6 Sol, helped refine the valid arguments and
exposition.  Some passages originated as AI-generated text, and much of the
manuscript received lighter AI-assisted editing.  The authors independently
verified and take full responsibility for all arguments and text.

\subsection{Tools}

Let $G$ be a graph, and let $u$ and $v$ be distinct vertices of $G$.  The \textit{Kelmans transformation of $G$ from $u$
to $v$} is the graph $G_{uv}$ obtained by deleting $ux$ and adding $vx$ for
every
\[
 x\in N(u)\setminus\bigl(N(v)\cup\{v\}\bigr);
\]
see~\cite{kelmans1981graphs} and Figure~\ref{fig:kelmans}. Note that for the transformation it does not matter if $uv$ was an edge or not.

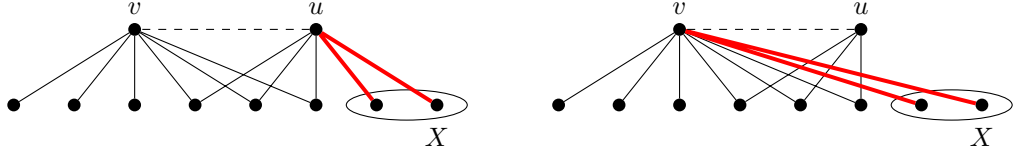
\begin{figure}
\begin{center}
    \begin{tikzpicture}[alt={Before transformation}, xscale=0.8]
        \foreach \x in {1,2,...,8}{
\draw (\x,0) node(\x)[unlabeled_vertex]{};
        }
\draw
(3,1) node[unlabeled_vertex,label=above:$v$](v){}
(6,1) node[unlabeled_vertex,label=above:$u$](u){}
(v)--(1)
(v)--(2)
(v)--(3)
(v)--(4)
(v)--(5)
(v)--(6)
(u)--(4)
(u)--(5)
(u)--(6)
;
\draw
(7.5,0) ellipse(1cm and 0.2cm)
(8)node[below=5pt]{$X$}
;
\draw[line width=1.5pt,color=red]
(u)--(7)
(u)--(8)
;
\draw[dashed](u)--(v);
    \end{tikzpicture}
\hskip 1cm
    \begin{tikzpicture}[alt={After transformation}, xscale=0.8]
        \foreach \x in {1,2,...,8}{
\draw (\x,0) node(\x)[unlabeled_vertex]{};
        }
\draw
(3,1) node[unlabeled_vertex,label=above:$v$](v){}
(6,1) node[unlabeled_vertex,label=above:$u$](u){}
(v)--(1)
(v)--(2)
(v)--(3)
(v)--(4)
(v)--(5)
(v)--(6)
(u)--(4)
(u)--(5)
(u)--(6)
;
\draw
(7.5,0) ellipse(1cm and 0.2cm)
(8)node[below=5pt]{$X$}
;
\draw[line width=1.5pt,color=red]
(v)--(7)
(v)--(8)
;
\draw[dashed](u)--(v);
    \end{tikzpicture}
\end{center}
\caption{The Kelmans transformation from $u$ to $v$.}
\label{fig:kelmans}
\end{figure}

Csikv\'ari proved the following monotonicity result.
\begin{theorem}[Csikv\'ari~\cite{Csikvari2011}, Theorem 5.10]\label{thm:KrIr}
For every pair of distinct vertices $u,v\in V(G)$ and every integer $r\ge1$,
 the Kelmans transformation $G_{uv}$ of $G$ from $u$ to $v$ satisfies
\[
 K_r(G_{uv})\ge K_r(G)\qquad\qquad\text{and}\qquad\qquad
 I_r(G_{uv})\ge I_r(G).
\]
\end{theorem}

A \textit{threshold graph} is a graph for which a vertex ordering $v_1, \dotsc, v_n$ with $d(v_1) \ge d(v_2) \ge  \ldots \ge d(v_n)$ satisfies
\[
N(v_j)\setminus\{v_i\}\subseteq N(v_i)\setminus\{v_j\}
\]
for every $1 \le i < j \le n$.

\begin{theorem}[Csikv\'ari~\cite{Csikvari2011}, Theorem 2.6(a)]\label{thm:order}
Every graph can be transformed into a threshold graph by a sequence of
Kelmans transformations.
\end{theorem}

Since each Kelmans transformation preserves the number of edges,
Theorems~\ref{thm:KrIr} and~\ref{thm:order} give the following.
\begin{corollary}\label{cor:threshold}
For all integers $r,m,n$ with $r,n\ge1$ and $0\le m\le\binom n2$, there is a
threshold graph $G$ on $n$ vertices with $m$ edges that maximizes
$K_r(G)+I_r(G)$ among all such graphs.
\end{corollary}

We use graphons and step graphons in our proofs; for further background and
the fundamental results of graph limit theory, see~\cite{lovasz2012large}.
A {\it graphon} is a symmetric measurable function
$W\colon[0,1]^2\to[0,1]$.  For $r\ge1$, we write
$t_{\mathrm{ind}}(K_r,W)$ and $t_{\mathrm{ind}}(I_r,W)$ for its induced
clique and independent-set densities, respectively.  Set
\[
 \Psi_r(W)=t_{\mathrm{ind}}(K_r,W)+t_{\mathrm{ind}}(I_r,W),
 \qquad\qquad
 \rho(W)=\int_{[0,1]^2}W(x,y)\,\mathrm{d}x\,\mathrm{d}y.
\]
Here $\rho(W)$ is the edge density of $W$.
A graphon $W$ is called a  {\it step graphon} if   there is a finite measurable partition
$[0,1]=P_1\cup\cdots\cup P_s$ such that $W$ is constant on every
$P_i\times P_j$.  The sets $P_i$ are called its {\it parts}.  A $0$--$1$ step
graphon is a step graphon whose block values all belong to $\{0,1\}$.
For a graph $G$ on the ordered vertex set $\{1,\ldots,n\}$, its associated
step graphon is obtained by taking $n$ equal consecutive parts and assigning
the value $1$ to $P_i\times P_j$ exactly when $ij\in E(G)$; its diagonal
blocks are assigned the value $0$.
We call $W$ \emph{nested} if it has a representative satisfying
\[
 W(x,z)\ge W(y,z)
 \quad\text{for almost every }(x,y,z)\text{ with }x\le y.
\]
For a nested step graphon, we always use the ordered representative in which
the sets of points having identical rows are consecutive intervals.  By
symmetry, every row of a nested $0$--$1$ step graphon is then the indicator
of an initial interval, up to a null set.  We call the resulting symmetric
monotone $0$--$1$ block pattern its \emph{staircase pattern}.

\section{Proof of Theorem~\ref{thm:main}}

Fix an integer $r\ge3$, and let
\[
 g_r(t):=t^r+(1-t)^r+rt(1-t)^{r-1}\qquad\mbox{and}\qquad
 M_r(\rho):=g_r\!\left(\sqrt{\max\{\rho,1-\rho\}}\right).
\]
Suppose, for the sake of contradiction, that there are $\eta>0$ and, for arbitrarily large integers $k$,
$k$-vertex graphs $G_k$ such that
\[
 \frac{K_r(G_k)+I_r(G_k)}{\binom kr}
 >M_r(\sigma_k)+\eta,
 \qquad \sigma_k=\rho(G_k)=\frac{e(G_k)}{\binom k2}.
\]
By Corollary~\ref{cor:threshold}, we may assume that  each $G_k$ is a threshold
graph and its  vertices are ordered so that  their degrees are non-increasing.  For the associated step graphon,
the  ordering implies that its diagonal blocks can be set to $1$
on an initial segment and to $0$ thereafter so that every column is
non-increasing.  The resulting nested
$0$--$1$ step graphon $V_k$ has at
most $k$ parts and differs from the associated graphon only on a set of area
at most $1/k$.  Put
$\rho_k=\rho(V_k)$.  The alteration of the diagonal blocks, together with
the difference between sampling vertices with and without replacement, gives
\[
 \rho_k=\sigma_k+O(1/k),\qquad
 \Psi_r(V_k)=\frac{K_r(G_k)+I_r(G_k)}{\binom kr}+O_r(1/k).
\]
Since $M_r$ is continuous on $[0,1]$, it follows that, for all sufficiently
large $k$ under consideration,
\[
 \Psi_r(V_k)>M_r(\rho_k)+\frac\eta2.
\]
For each such $k$, choose a
nested $0$--$1$ step graphon $W_k$ with at most $k$ parts and edge density
$\rho_k$ that maximizes $\Psi_r$ in this finite-dimensional class.  Such a
maximizer exists because there are finitely many staircase patterns and,
allowing zero-length parts, each pattern has a compact set of admissible block
lengths on which $\Psi_r$ is continuous.  Since $V_k$
belongs to this class,
\[
 \Psi_r(W_k)\ge\Psi_r(V_k)>M_r(\rho_k)+\frac\eta2.
\]
We will prove below that every such maximizer satisfies
$\Psi_r(W_k)\le M_r(\rho_k)$, which is the desired contradiction.

For $x\in[0,1]$, let
\[
 d_k(x)=\int_0^1W_k(x,y)\,\mathrm{d}y
\]
be its graphon degree.  In the ordered nested representative, the neighbors
of $x$ form the initial interval $[0,d_k(x))$, up to a null set.  Thus the
measure of its neighbors preceding it is $\min\{x,d_k(x)\}$.  Counting each
edge by its later endpoint and recalling that graphon edge density counts
ordered pairs gives
\begin{equation}\label{eq:edgedensity}
\rho_k=2\int_0^1\min\{x,d_k(x)\}\,\mathrm{d}x.
\end{equation}
Similarly, an $r$-clique is counted by its last point and an independent
$r$-set by its first point.  Nestedness ensures that the other $r-1$ points
form a clique or an independent set, respectively.  Therefore
\[
 t_{\mathrm{ind}}(K_r,W_k)
 =r\int_0^1\min\{x^{r-1},d_k(x)^{r-1}\}\,\mathrm{d}x
\]
and
\[
 t_{\mathrm{ind}}(I_r,W_k)
 =r\int_0^1\min\{(1-x)^{r-1},(1-d_k(x))^{r-1}\}\,\mathrm{d}x.
\]
Consequently, subject
to~\eqref{eq:edgedensity}, $W_k$ maximizes
\begin{equation}\label{eq:KrIrdensity}
r\int_0^1\min\{x^{r-1},d_k(x)^{r-1}\}\,\mathrm{d}x
+r\int_0^1\min\{(1-x)^{r-1},(1-d_k(x))^{r-1}\}\,\mathrm{d}x.
\end{equation}

After merging consecutive intervals with identical rows, call the resulting
intervals the canonical parts.  Suppose that $W_k$ has $s\ge4$ such parts.
If its last row is nonzero, then that row begins with a $1$; symmetry and
nestedness force the first row to be identically $1$.  Thus either $W_k$
already has a zero last row or its complement followed by order reversal does.
Let $\widehat W_k$ be the resulting graphon, and let $\widehat\rho_k$ be its
edge density.  Then
\[
 \widehat\rho_k\in\{\rho_k,1-\rho_k\},\qquad
 \Psi_r(\widehat W_k)=\Psi_r(W_k).
\]
Moreover, $\widehat W_k$ is a maximizer in the corresponding class at density
$\widehat\rho_k$, and formulas~\eqref{eq:edgedensity}--\eqref{eq:KrIrdensity}
apply with $W_k,\rho_k$ replaced by $\widehat W_k,\widehat\rho_k$.

For $i\in\{1,\ldots,s\}$, let $t_i$ be the number of canonical parts on
which the $i$th row of $\widehat W_k$ equals $1$.  Nestedness and distinctness
give
\[
 t_1>t_2>\cdots>t_s=0.
\]
Since the last row is zero, symmetry makes the first row zero on the last
part, so $t_1\le s-1$.  Hence necessarily $t_i=s-i$ for every $i$.  In
particular, the first column is $1$ except on the last part, and the second
column is $1$ except on the last two parts.  Let $a,b$ be the lengths of the
first two parts and $c,d$ the lengths of the last two.  These lengths are
positive and $a+b+c+d\le1$.

With these parameters, the four forced outer blocks are
\[\widehat W_k(x,y) = \begin{cases}
1 \quad \quad\quad& \text{if } x \in [0,1-d) \text{ and } y \in [0,a), \\
0 & \text{if } x \in [1-d,1] \text{ and } y \in [0,a), \\
1 & \text{if } x \in [0,1-c-d) \text{ and } y \in [a,a+b), \\
0 & \text{if } x \in [1-c-d,1] \text{ and } y \in [a,a+b). \\
\end{cases}\]
\begin{figure}[ht]
\begin{center}
  \includegraphics[alt={A graphon with adjacencies of first 2 blocks a,b and last two blocks c,d indicated but the rest is not forced.}]{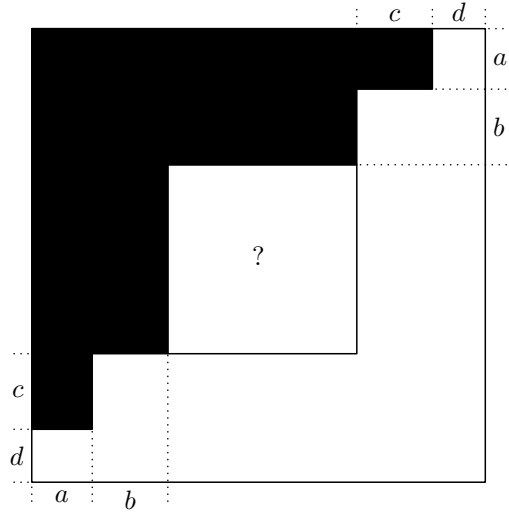}
\caption{The forced corner blocks of the ordered nested graphon
$\widehat W_k$.}
\end{center}
\end{figure}

The contribution of the vertices
$x\in[0,a+b)\cup[1-c-d,1]$ to the right-hand side
of~\eqref{eq:edgedensity} is $(a+b)^2+2ac$, while their contribution
to~\eqref{eq:KrIrdensity} is
\begin{equation}\label{eq:KrIrcontribution}
(a+b)^r+rca^{r-1}+rad^{r-1}+rb(c+d)^{r-1}
+(c+d)^r.
\end{equation}

Let
\begin{equation*}
f(w,x,y,z):=(w+x)^r+ryw^{r-1}+rwz^{r-1}
+rx(y+z)^{r-1}+(y+z)^r.
\end{equation*}
For any non-negative quadruple $(w,x,y,z)$ satisfying
\[
 w+x=a+b,\qquad y+z=c+d,
 \qquad (w+x)^2+2wy=(a+b)^2+2ac,
\]
retain the same staircase pattern and all intervening parts, but give the
first two and last two parts lengths $w,x,y,z$.  The first two equalities keep
the total outer length, and hence all intervening parts, fixed.  The third
equality preserves the outer contribution to the edge density.  All other
terms in~\eqref{eq:KrIrdensity} are unchanged, while
equation~\eqref{eq:KrIrcontribution} shows that the varying part is exactly
$f$.
Thus every feasible quadruple is realized by an admissible nested step
graphon with the same total edge density.

Because $\widehat W_k$ maximizes~\eqref{eq:KrIrdensity} subject
to~\eqref{eq:edgedensity}, the tuple $(a,b,c,d)$ maximizes $f$ over all such
quadruples.

Put
\[
 p=a+b,\qquad q=c+d,\qquad \alpha=p^2+2ac.
\]
Since $a,b,c,d>0$ and $p+q\le1$, we have $p,q\in(0,1)$.  Moreover,
$a<p$ and, because $d>0$, $c<q\le1-p$.  Therefore
\[
 p^2<\alpha<p^2+2p(1-p)=1-(1-p)^2<1.
\]
Thus all the domain assumptions of Lemma~\ref{lem:maximizer4} are satisfied.

Lemma~\ref{lem:maximizer4} now implies that at least one coordinate of the
maximizer $(a,b,c,d)$ is zero, contradicting $a,b,c,d>0$.

Therefore, every $W_k$ has at most three canonical parts.  Each $W_k$ has one
of the two structures depicted in Figure~\ref{fig:graphonW}, with parameters
$p,q\ge0$ satisfying $p+q\le1$.

\begin{figure}[ht]
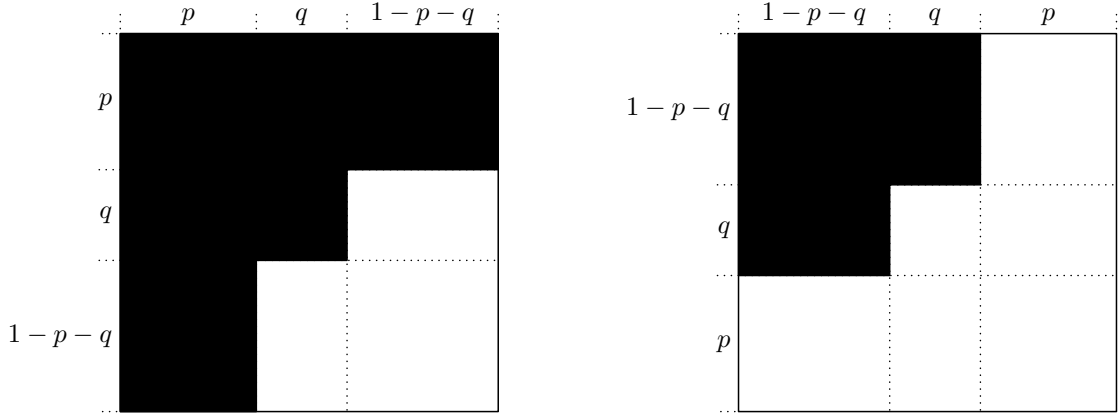

\begin{center}
\includegraphics[alt={A step graphon 3x3 with the right bottom under diagonal empty and rest full. Blocks have sizes p, q, 1-p-q.}]{graphon4}\hspace{15mm}
\includegraphics[alt={A step graphon 3x3 with the left top above diagonal filled and rest empty. Blocks have sizes 1-p-q, q, p.}]{graphon3}
\caption{The possible structures of a maximizing graphon $W_k$.}\label{fig:graphonW}
\end{center}
\end{figure}

For both structures we have
\[
 \Psi_r(W_k)=(p+q)^r+r(1-p-q)p^{r-1}
 +rq(1-p-q)^{r-1}+(1-p-q)^r,
\]
and
\[
 e=(p+q)^2+2p(1-p-q)
\]
is either $\rho_k$ or $1-\rho_k$.

If $e\in\{0,1\}$, then $W_k$ is respectively empty or complete up to
complementation, and $\Psi_r(W_k)=1=M_r(\rho_k)$.  Suppose that $0<e<1$.
 Then Lemma~\ref{lem:maximizer3_domain} gives
\[
 \Psi_r(W_k)\le
 \max_{\gamma\in\{\sqrt e,\sqrt{1-e}\}}g_r(\gamma).
\]
The two possible values of $\gamma$ have squares summing to $1$.
Lemma~\ref{lem:maximizer3_value} therefore shows that their maximum is
attained at $\gamma=\sqrt{\max\{e,1-e\}}$.  Since
$\{e,1-e\}=\{\rho_k,1-\rho_k\}$, we obtain
\[
 \Psi_r(W_k)\le M_r(\rho_k).
\]
This contradicts the strict inequality
$\Psi_r(W_k)>M_r(\rho_k)+\eta/2$ obtained above.  Hence the uniform upper
bound holds, completing the proof of Theorem~\ref{thm:main}.

\section{Auxiliary optimization (AI-assisted proofs)}

The proofs in this section contain the optimization arguments.  Section~\ref{aicomp}
contains the differentiations and algebraic identities used here, separating
those routine computations from the sign, compactness, and endpoint arguments.

\begin{lemma}\label{lem:maximizer4}
  Fix an integer $r \ge 3$ and three positive reals $p,q$ and $\alpha$ that satisfy the conditions
  $p+q\le1$ and $p^2 < \alpha < 1-(1-p)^2$, and let
  \[f(w,x,y,z) := (w+x)^r + ryw^{r-1} + rwz^{r-1} + rx(y+z)^{r-1} + (y+z)^r.\]
  If $(w^*, x^*, y^*, z^*)$ is a global maximizer of the program
  \[
  (P)=\begin{cases}
  \text{maximize} &
  f(w,x,y,z)
  \\
  \text{subject to} &
  w+x=p \\
  & y+z=q \\
  & (w+x)^2+2wy=\alpha \\
  &(w,x,y,z) \geq 0,
  \end{cases}
  \]
  then at least one of $w^*$, $x^*$, $y^*$, and $z^*$ is equal to zero.
\end{lemma}

\begin{proof}
Let $C=(\alpha-p^2)/2$.  On the constraint set we have
\[
 x=p-w,\qquad z=q-y,\qquad C= (\alpha-p^2)/2= \left((w+x)^2+2wy -(w+x)^2\right)/2= wy.
\]
Suppose for contradiction  that all four coordinates of a maximizer are positive.  Then
$0<C=wy<pq$, and
\[
 (w,x,y,z)=\left(w,p-w,\frac Cw,q-\frac Cw\right),
 \qquad \frac Cq < w < p,
\]
since the values $w=C/q$ and $w=p$ correspond to $z=0$ and $x=0$, respectively.
On this curve, the objective is to maximize
\[
 f\left(w,p-w,\frac Cw,q-\frac Cw\right)
=p^r + r\cdot\frac{C}{w}\cdot w^{r-1} + rw\left(q-\frac Cw\right)^{r-1} + r(p-w)q^{r-1} + q^r =p^r+q^r+rpq^{r-1}+rh(w),
\]
where
\[
h(w)=Cw^{r-2}+w\left(q-\frac Cw\right)^{r-1}-wq^{r-1}.
\]
Thus, maximizing the objective is equivalent to maximizing $h(w)$ on the open interval $p > w > C/q$.
However, Lemma~\ref{lem:calc_four_block} shows that $h$ is strictly convex in this interval,
so it cannot attain its maximum there; a contradiction.
\end{proof}

\begin{lemma}\label{lem:maximizer3_domain}
  Fix an integer $r \ge 3$ and a real $e \in(0,1)$, and let
  \[F(p,q) := (p+q)^r + r(1-p-q)p^{r-1} + rq(1-p-q)^{r-1} + (1-p-q)^r.\]
  If $(p^*,q^*)$ is a global maximizer of the program
  \[
  (P)=\begin{cases}
  \text{maximize} &
  F(p,q)
  \\
  \text{subject to} &
  p+q\le1 \\
  & (p+q)^2+2p(1-p-q)=e \\
  &(p,q) \geq 0,
  \end{cases}
  \]
  then $(p^*,q^*) \in \left\{ (0,\sqrt e), (1-\sqrt{1-e},0)\right\}$.
\end{lemma}
\begin{proof}
Put $y=1-p-q$.  Since $e<1$, every feasible point has $y>0$ (note that $y=0$ yields $e=1$).
The constraints $e=(1-y)^2+2py$ and $p+q=1-y$ yield the following re-parametrization
\[
 p(y)=\frac{e-(1-y)^2}{2y},\qquad
 q(y)=1-p-y=\frac{2y-2py-2y^2}{2y}= \frac{1-(1-y)^2-2py-y^2}{2y}=\frac{1-e-y^2}{2y}.
\]
The conditions $p,q\ge0$ are equivalent to
\[
 1-\sqrt e\le y\le\sqrt{1-e}.
\]
Thus the feasible curve $(p(y),q(y))$ is a  continuous image of a closed interval and is therefore compact.
Its endpoints are
\[
 (p,q)=(0,\sqrt e)\quad\quad\quad\text{and}\quad\quad\quad
 (p,q)=(1-\sqrt{1-e},0).
\]

The function $F(p,q)$ is continuous, so it attains a maximum on the feasible set.  To locate
that maximum, define
\[
 \Phi(y)=F(p(y),q(y))
 =(1-y)^r+ry\cdot p(y)^{r-1}+rq(y) \cdot y^{r-1}+y^r.
\]
Thus maximizing $(P)$ is equivalent to maximizing $\Phi(y)$ on the interval $[1-\sqrt{e},\sqrt{1-e}]$.
Since $\Phi''(y)>0$ for all $y$ such that $\Phi'(y)=0$ by Lemma~\ref{lem:calc_three_block}, every critical point of $\Phi$ inside the interval is a strict local minimum.
Therefore, a global maximum of $\Phi$ occurs at one of the endpoints.
\end{proof}

\begin{lemma}\label{lem:circle_crossing}
Fix an integer $r\ge3$, and let
\[
 f(t):=t^{r-2}-(r-1)(1-t)^{r-2}\,,
 \qquad
 t_0:=\frac{(r-1)^{1/(r-2)}}{1+(r-1)^{1/(r-2)}}
 \qquad\mbox{and}\qquad 
  D(s):=f\left(\sqrt{1-s^2}\right)+f(s).
\]
Then $D(0)<0$, $D(t_0)>0$ and $D: [0,t_0]\to \mathbb{R}$ has exactly one zero.
\end{lemma}
\begin{proof}
We have $D(0)=f(1)+f(0)=1 -(r-1)=2-r<0$.
Lemma~\ref{lem:calc_G_shape} shows that $f$ is strictly increasing on $[0,1]$,
$f(t_0)=0$, and $t_0<1/\sqrt2$.  Therefore
$\sqrt{1-t_0^2}>t_0$, and
\[
 D(t_0)
 =f\left(\sqrt{1-t_0^2}\right)+f(t_0)
 =f\left(\sqrt{1-t_0^2}\right)
 >f(t_0)=0.
\]
Let
\[
 Z=\{u\in[0,t_0]:D(u)=0\}.
\]
Since $D$ is a continuous map on $[0,t_0]$, the opposite signs at the endpoints and the intermediate value theorem show that $Z$ is nonempty.
Moreover, every $u\in Z$ satisfies the assumptions of Lemma~\ref{lem:calc_circle_crossing} and thus $D'(u)>0$.
Finally, the set $D^{-1}(\{0\})$ is a closed subset of $[0,t_0]$, and hence $\inf Z \in Z$. 
Let $u_1=\min Z$.  The derivative at $u_1$ shows that $D$ is positive on
$(u_1,u_1+\delta)$ for some $0<\delta<t_0-u_1$.  If $Z\setminus \{u_1\}\neq\emptyset$, then the
nonempty closed set $Z\cap[u_1+\delta,t_0]$ would have a minimum $u_2$.
There is no zero of $D$ in $(u_1,u_2)$, so the intermediate value theorem yields that $D>0$ throughout this interval.
However, $D'(u_2)>0$ yields that $D$ must be negative slightly to the left of $u_2$; a contradiction.
We conclude that $u_1$ is the unique zero of $D$.
\end{proof}

\begin{lemma}\label{lem:maximizer3_value}
Fix an integer $r\ge3$, and let
\[
 g_r(t):=t^r+(1-t)^r+rt(1-t)^{r-1}.
\]
If $x$ and $y$ are non-negative reals satisfying $x\ge y$ and $x^2+y^2=1$, then $g_r(x)\ge g_r(y)$.
\end{lemma}
\begin{proof}
Let
\[
 f(t):=t^{r-2}-(r-1)(1-t)^{r-2}.
\]
Lemma~\ref{lem:calc_G_shape} states that $f$ has a unique zero $t_0$ where
$t_0<1/\sqrt2$, and that $g_r(t)$ decreases when $t<t_0$ and increases when
$t>t_0$.  The assumption $x^2+y^2=1$ implies
\[
 0\le y\le\frac1{\sqrt2}\le x\le1.
\]
If $y\ge t_0$, then as both $x,y$  lie in the increasing region of
$g_r$ and $x\ge y$, we have  $g_r(x)\ge g_r(y)$.

It remains to handle  the case $y<t_0$. 
For $s\in[0, t_0]$, set
\[
 H(s)=g_r\left(\sqrt{1-s^2}\right)-g_r(s),\qquad
 D(s)=f\left(\sqrt{1-s^2}\right)+f(s).
\]
Our aim is to prove that $H(s)\ge0$ for $s \in [0,t_0]$.
Using $g_r'(t)=rtf(t)$ from Lemma~\ref{lem:calc_G_shape}, the chain rule
gives
\begin{align*}
 H'(s)
 &=-\frac{s}{\sqrt{1-s^2}}\cdot g_r'\left(\sqrt{1-s^2}\right)-g_r'(s)
 =-rsf\left(\sqrt{1-s^2}\right)-rsf(s)
 =-rsD(s).
\end{align*}
Thus, for all $s\in(0,t_0]$, the derivative $H'(s)$ has the opposite sign from
$D(s)$.  Let $s_0$ be the unique zero of $D$ given by
Lemma~\ref{lem:circle_crossing}.  
The conclusion of Lemma~\ref{lem:circle_crossing} translates to
\[
 H'(s)>0\quad\forall s\in(0,s_0),\quad \mbox{and}\qquad
 H'(s)<0\quad\forall s\in(s_0,t_0).
\]
Therefore, $H$ is strictly increasing on $[0,s_0]$ and strictly
decreasing on $[s_0,t_0]$.
Consequently, the minimum of $H$ on $[0,t_0]$ is attained at one of the  endpoints of the interval.
Clearly,
\(
 H(0)=g_r(1)-g_r(0)=0
\).
On the other hand, we claim that $H(t_0)>0$.
Indeed, since $\sqrt{1-t_0^2}>t_0$ and $g_r(x)$ is strictly increasing for $x\ge t_0$,
it holds that
\[
 H(t_0)=g_r\left(\sqrt{1-t_0^2}\right)-g_r(t_0)>0.
\]
We conclude that $H(s)\ge0$ on $[0,t_0]$, which finishes the proof.
\end{proof}

\section{Routine computations (AI-assisted proofs)}\label{aicomp}

In this section we present the computational parts used in our proofs.

\begin{lemma}\label{lem:calc_four_block}
Fix an integer $r\ge3$ and two positive reals $q$ and $C$.
Let
\[
 h(w):=Cw^{r-2}+w\left(q-\frac Cw\right)^{r-1}-wq^{r-1}.
\]
If $w>C/q$ then $h''(w)>0$.
\end{lemma}
\begin{proof}
 The product and chain rules give
 \[
  h'(w)=(r-2)Cw^{r-3}
  +\left(q-\frac Cw\right)^{r-1}
  +(r-1)\frac Cw\left(q-\frac Cw\right)^{r-2}
  -q^{r-1}.
 \]
 Differentiating the second term gives
 \[
  (r-1)\frac{C}{w^2}\left(q-\frac Cw\right)^{r-2},
 \]
 whereas differentiating the third term gives
 \[
  -(r-1)\frac{C}{w^2}\left(q-\frac Cw\right)^{r-2}
  +(r-1)(r-2)\frac{C^2}{w^3}
  \left(q-\frac Cw\right)^{r-3}.
 \]
 The first term here cancels the term obtained above; thus,
\[
 h''(w)=(r-2)(r-3)C\cdot w^{r-4}
 +(r-1)(r-2)\cdot \frac{C^2}{w^3}\cdot\left(q-\frac Cw\right)^{r-3}.
\]
The first summand is non-negative for all $r\ge3$, and thus it remains to show that the second summand is strictly positive.
Since
\[
 \frac{w\left(q-\frac Cw\right)}q=w-\frac Cq>0,
\]
we conclude that the term $\left(q-\frac Cw\right)$ is positive, which together with $w>0$, $C>0$ and $r>2$ concludes the proof.
\end{proof}

\begin{lemma}\label{lem:calc_three_block}
Fix an integer $r\ge3$ and a real number $e\in(0,1)$.
For   $y \in \left(1-\sqrt e,\sqrt{1-e}\right)$,
let
\[
 p=p(y):=\frac{e-(1-y)^2}{2y},\qquad
 q=q(y):=\frac{1-e-y^2}{2y},\qquad
 x=x(y):=1-y=p+q,
\]
and
\[
 \Phi(y):=x^r+ryp^{r-1}+rqy^{r-1}+y^r.
\]
If $\Phi'(y)=0$, then $\Phi''(y)>0$.
\end{lemma}
\begin{proof}
Inside the  open interval  $\left(1-\sqrt e,\sqrt{1-e}\right)$, we have  $p,q,y>0$, and
\[
 p'=p'(y)=\frac qy,\qquad q'=q'(y)=\frac{p-1}{y},\qquad x'=x'(y)=-1.
\]

 Indeed, rewriting $p$ and $q$ gives
 \[
  p=\frac{e-1}{2y}+1-\frac y2,
  \qquad
  q=\frac{1-e}{2y}-\frac y2.
 \]
 Differentiating then gives
 \[
  p'=-\frac{e-1}{2y^2}-\frac12=\frac qy,
  \qquad
  q' =-\frac{1-e}{2y^2}-\frac12=\frac{p-1}{y}.
 \]
 
 Differentiating the four summands of $\Phi$ separately gives
 \[
  \Phi'(y)=r\left[x^{r-1}x'+p^{r-1}+(r-1)yp^{r-2}p'+q'y^{r-1}+(r-1)qy^{r-2}+y^{r-1}\right].
 \]
 Substituting $x'=-1$, $p'=q/y$, and $q'=(p-1)/y$ produces the first
 expression below
\begin{equation}\label{eq:calc_phi_prime}
\begin{split}
 \Phi'(y)&=
  r\left[ -x^{r-1}+p^{r-1}+(r-1)qp^{r-2} +(p-1)y^{r-2}+(r-1)q\cdot y^{r-2}+ y^{r-1}\right] \\
  &=
 r\left[p^{r-1}+(r-1)qp^{r-2}
 +(r-2)qy^{r-2}-x^{r-1}\right],
 \end{split}
  \end{equation}
and for the second equality we used $p+q-1=-y$.

At a critical point of $\Phi$, i.e., when 
$ \Phi'(y)=0$, substituting $x=p+q$ into~\eqref{eq:calc_phi_prime} gives
\[
 (r-2)qy^{r-2}
 =(p+q)^{r-1}-p^{r-1}-(r-1)qp^{r-2}.
\]
Expanding $(p+q)^{r-1}$ and dividing by $q>0$ yield that when $ \Phi'(y)=0$, then
\begin{equation}\label{eq:calc_critical_condition}
 (r-2)y^{r-2}
 =\sum_{k=2}^{r-1}\binom{r-1}{k}p^{r-1-k}q^{k-1}.
\end{equation}

Now we  compute the second derivative of $\Phi(y)$: 
\begin{align*}
 \Phi''(y)
 &= r(r-1)q\cdot\frac{p^{r-2}}{y} + r(r-1)(p-1)\frac{p^{r-2}}{y} + r(r-1)(r-2)p^{r-3}\cdot\frac{q^2}{y} \\
 &  \qquad + r(r-2)(p-1)y^{r-3} + r(r-2)^2qy^{r-3} + r(r-1)x^{r-2}
\,.
\end{align*}
Using the identity $p + q -1 = -y$ twice in the following reordering
\begin{align*}
 \Phi''(y)
 &= r(r-1)\left(q+p-1\right)\cdot \frac{p^{r-2}}{y} +  r(r-1)(r-2)p^{r-3}\cdot\frac{q^2}{y} \\
 &  \qquad + r(r-2)(p-1+q)y^{r-3} + r(r-2)(r-3)qy^{r-3} + r(r-1)x^{r-2},
\,
\end{align*}
we obtain
\begin{equation}\label{doubleder}
 \Phi''(y)
 =r(r-1)\left(x^{r-2}-p^{r-2}\right)
   +r(r-1)(r-2)p^{r-3}\cdot\frac{q^2}{y}
 +r(r-2)(r-3)qy^{r-3}-r(r-2)y^{r-2}.
\end{equation}

First, note that the second summand of $\Phi''$ is positive and the third summand is non-negative.
Next, using $x = p+q$, the first summand of $\Phi''$ can be rewritten as follows:
\begin{equation}\label{eq:calc_phi_first}
 r(r-1)(x^{r-2}-p^{r-2})
 =r(r-1)\sum_{j=1}^{r-2}\binom{r-2}{j}p^{r-2-j}q^j.
\end{equation}
Finally, at every critical point of $\Phi$, i.e., when 
$ \Phi'(y)=0$,
the last summand of $\Phi''$ can be rewritten using~\eqref{eq:calc_critical_condition} in the following way: 
\begin{equation}\label{eq:calc_phi_last}
 -r(r-2)y^{r-2}
 =-r\sum_{j=1}^{r-2}\binom{r-1}{j+1}p^{r-2-j}q^j.
\end{equation}
We claim that the sum of \eqref{eq:calc_phi_first} and \eqref{eq:calc_phi_last} is strictly positive.
Indeed, using  
\[
 (r-1)\binom{r-2}{j}=(j+1)\binom{r-1}{j+1}
\]
gives that in~\eqref{doubleder} the coefficient in front of $p^{r-2-j}q^j$,  for  $j \in [r-2]$, is equal to 

\begin{align*}
 r(r-1)\binom{r-2}{j}-r\binom{r-1}{j+1}
 &=r\left((j+1)-1\right)\binom{r-1}{j+1}
 =rj\binom{r-1}{j+1}>0.
\end{align*}

Therefore, we conclude that $\Phi''(y)>0$ at every critical point of $\Phi$, i.e., when 
$ \Phi'(y)=0$.
\end{proof}

\begin{lemma}\label{lem:calc_G_shape}
Fix an integer $r\ge3$, and let
\[
 g_r(t)=t^r+(1-t)^r+rt(1-t)^{r-1} \qquad\mbox{and}\qquad
 f(t)=t^{r-2}-(r-1)(1-t)^{r-2}.
\]
It holds that $g_r'(t)=rtf(t)$, $f$ is strictly increasing on $[0,1]$, and
$f$ has the unique zero
\[
 t_0=\frac{(r-1)^{1/(r-2)}}{1+(r-1)^{1/(r-2)}}<\frac1{\sqrt2}.
\]
In particular, $g_r$ is strictly decreasing on $[0,t_0]$ and strictly increasing on
$[t_0,1]$.
\end{lemma}
\begin{proof}
Direct computation of the derivatives provides
\[
  g_r'(t)=r\cdot t^{r-1} - r\cdot (1-t)^{r-1} +
 r\cdot (1-t)^{r-1}-
 r\cdot t\cdot (r-1) (1-t)^{r-2}= r\cdot t\cdot f(t), \qquad\mbox{and}\]
 \[
 f'(t)=(r-2)\left[t^{r-3}+(r-1)(1-t)^{r-3}\right].
\]
Since $f(0)=1-r <0<f(1)=1$ and $f'(t)>0$ for $t\in [0,1]$, $f$ has exactly one zero in $[0,1]$.
Solving 
\[
 \left(\frac{t}{1-t}\right)^{r-2}=r-1
\]
yields the displayed formula for $t_0$.  To prove the claimed upper bound $1/\sqrt2$, note
that $r-1\le2^{r-2}$ for every $r\ge3$.
Thus $(r-1)^{1/(r-2)}\le2$, and hence
$t_0\le2/3<1/\sqrt2$.

 Finally, because $rt>0$ for $t\in(0,1]$, the identity
 $g_r'(t)=rtf(t)$ shows that $g_r'$ has the same sign as $f$.  Since $f$ is
 strictly increasing and vanishes only at $t_0$, we have
 \[
  g_r'(t)<0\quad\text{when}\ \ 0<t<t_0,
  \qquad\qquad\text{and}\qquad\qquad
  g_r'(t)>0\quad\text{when}\ \ t_0<t\le1,
 \]
 which gives the asserted monotonicity of $g_r$.
\end{proof}

\begin{lemma}\label{lem:calc_circle_crossing}
Fix a positive integer $n$, and let
\[
 f(t)=t^{n}-(n+1)(1-t)^{n},
 \qquad
 D(s)=f\left(\sqrt{1-s^2}\right)+f(s).
\]
For every $u \in (0,1/\sqrt2)$ such that $D(u)=0$, it holds that $D'(u)>0$.
\end{lemma}
\begin{proof}
To simplify the presentation of the proof, we introduce some auxiliary notation.
Let $v:=\sqrt{1-u^2}$, and note that $0<u<v<1$.
Also, let
\[
 A:=v\cdot\big(u(1-u)\big)^{n-1}+v^{n-1}\cdot(1-u)^n \qquad \qquad\mbox{and}\qquad\qquad
 B:=u\cdot\big(v(1-v)\big)^{n-1}+u^{n-1}\cdot(1-v)^n,
\]
which will be used to simplify the expression $v\cdot f'(u) - u\cdot f'(v)$.

\begin{claim}
 $A>B$.
\end{claim}
\begin{proof}
Indeed, since $v>u$ and $1-u>1-v$, it is enough to observe that $u(1-u) > v(1-v)$.
 The relation $u^2+v^2=1$, with $0< u< v<1 $ and $u+v\ge 1$ gives that $|u-0.5|< |v-0.5|$, which implies  $u(1-u) > v(1-v)$.
\end{proof}

Next,
{since
\[
 \frac{d}{ds}\sqrt{1-s^2}=-\frac{s}{\sqrt{1-s^2}},
\]
the chain rule gives
\[
D'(s)=f'(s)-\frac{s}{\sqrt{1-s^2}} \cdot
 f'\left(\sqrt{1-s^2}\right).
\]
Evaluating at $s=u$ and using $\sqrt{1-u^2}=v$ yields $D'(u)=f'(u)-\frac uv f'(v)$.
}
Therefore, in our quest to prove $D'(u)>0$ under the assumption $D(u)=f(u)+f(v)=0$,
it suffices to establish the following identity:
\begin{equation}\label{eq:calc_circle_crossing_goal}
\frac{v}{n}\cdot D'(u)= \frac{v\cdot f'(u)-u\cdot f'(v)}n = \frac{A-B}{(1-u)^n+(1-v)^n} >0,
\end{equation}
where  the inequality holds due to the fact $A>B$.

Observe that evaluating $f(t)$ for $t=u$ and $t=v$ and rearranging yields that
\[
 f(u)+f(v)=0 \quad \iff \quad n+1=\frac{u^n+v^n}{(1-u)^n+(1-v)^n}.
\]
Moreover, a direct computation of the derivative of $f$ shows that
\[
f'(t) = n\cdot \left(t^{n-1}+(n+1)(1-t)^{n-1}\right).
\]
Therefore, a tedious yet straightforward calculation reveals that
\begin{align*}
    \frac{v\cdot f'(u)-u\cdot f'(v)}n
    & =v \left(u^{n-1}+(n+1)(1-u)^{n-1}\right) - u \left(v^{n-1}+(n+1)(1-v)^{n-1}\right)\\
    & =vu^{n-1} - uv^{n-1}+(n+1)\left(v(1-u)^{n-1} - u(1-v)^{n-1}\right)\\
    & =\frac{\left(vu^{n-1} - uv^{n-1}\right) \cdot \big({(1-u)^n+(1-v)^n}\big) + \left(u^n+v^n\right)\cdot \big(v(1-u)^{n-1} - u(1-v)^{n-1}\big)}{(1-u)^n+(1-v)^n}
    \,.
\end{align*}
Therefore, our goal of establishing \eqref{eq:calc_circle_crossing_goal} reduces to showing the following two identities for $A$ and $B$:
\begin{align*}
A&=\left(vu^{n-1} - uv^{n-1}\right)\cdot (1-u)^n + \left(u^n+v^n\right)\cdot v(1-u)^{n-1} \qquad \mbox{and}\\
-B&=\left(vu^{n-1} - uv^{n-1}\right) \cdot {(1-v)^n} - \left(u^n+v^n\right) \cdot u(1-v)^{n-1}.
\end{align*}

Let us first verify the identity for $A$.
Plugging in the definition of $A=v\cdot\big(u(1-u)\big)^{n-1}+v^{n-1}\cdot(1-u)^n$,
and then dividing both sides of the identity in question by the factor $v(1-u)^{n-1}$,
brings us to check the following (equivalent) identity:
\begin{equation}\label{eq:calc_circle_crossing_A}
u^{n-1} + (1-u)v^{n-2}
= \left(u^{n-1} - uv^{n-2}\right)\cdot(1-u) + u^n+v^{n}
= u^{n-1}-uv^{n-2}+\left(u^2+v^2\right)\cdot v^{n-2}.
\end{equation}
However, \eqref{eq:calc_circle_crossing_A} readily follows from  $u^2+v^2=1$.

Similarly, we verify the identity for $-B$.
Recalling that $B=u\cdot\big(v(1-v)\big)^{n-1}+u^{n-1}\cdot(1-v)^n$, and then dividing
both sides by the factor $-u(1-v)^{n-1}$, brings us to the identity
\begin{equation}\label{eq:calc_circle_crossing_B}
v^{n-1} + (1-v)u^{n-2}=u^{n}+v^n + \left(v^{n-1}-vu^{n-2} \right)\cdot{(1-v)}
= v^{n-1}-vu^{n-2}+\left(u^2+v^2\right)\cdot u^{n-2}
\,.
\end{equation}
As in the case of \eqref{eq:calc_circle_crossing_A}, the identity \eqref{eq:calc_circle_crossing_B}  now follows from $u^2+v^2=1$, which finishes the proof.
\end{proof}

{\bf Acknowledgments}
The authors are grateful to the American Institute of Mathematics for hosting  the workshop
``Flag algebras and extremal combinatorics'', where the project was started. Additionally, the authors are grateful to the organizers of the workshop ``Flags in Mountains'', where this work was continued. 
The work of the third and the sixth authors was also partially supported by the MEYS of the Czech Republic under the INTER-EXCELLENCE II program (project No. LUAUS26294).

\end{document}